\documentclass[12pt,reqno]{amsart}
\usepackage{amssymb,amsthm,amsmath,amstext,amsxtra}
\usepackage{mathrsfs,bm}
\usepackage{fullpage}
\usepackage[all]{xy}
\usepackage{booktabs}
\usepackage{mathtools}
\usepackage{hyperref}
\hypersetup{colorlinks=true,urlcolor=blue,citecolor=blue,linkcolor=blue}
\usepackage{enumerate}
\usepackage{ stmaryrd }
\usepackage{enumitem}
\usepackage{comment}
\usepackage{multirow}
\usepackage{colonequals}
\usepackage{tikz}
\usepackage{marginnote}
\usepackage[export]{adjustbox}

\usepackage{tikz}
\usepackage{tikz-cd}
\usepackage{xcolor}

\numberwithin{equation}{subsection}

\newtheorem{theorem}{Theorem}[section]
\newtheorem{lemma}[theorem]{Lemma}

\newtheorem{conjecture}[theorem]{Conjecture}
\newtheorem{remark}[theorem]{Remark}

\begin{document}

\title{Periods of CY $n$-folds and mixed Tate motives, a numerical study}

\author{Wenzhe Yang}
\address{SITP Stanford University, CA, 94305}
\email{yangwz@stanford.edu}

\begin{abstract}
In the mirror symmetry of Calabi-Yau threefolds, the instanton expansion of the prepotential has a constant term that is a rational multiple of $\zeta(3)/(2 \pi i)^3$, the motivic origin of which has been carefully studied in the author's paper with M. Kim. The Gamma conjecture claims that higher zeta values in fact appear in the theory of Calabi-Yau $n$-folds for $n \geq 4$. Moreover, a natural question is whether they also have a motivic origin. In this paper, we will study the limit mixed Hodge structure (MHS) of the Fermat pencil of Calabi-Yau $n$-folds at the large complex structure limit, which is actually a mixed Hodge-Tate structure. We will compute the period matrix of this limit MHS for the cases where $n=4,5,6,7,8,9,10,11,12$ by numerical method, and our computations have shown the occurrence of higher zeta values, which provide evidence to the Gamma conjecture. Furthermore, we will also provide a motivic explanation to our numerical results.
\end{abstract}

\maketitle
\setcounter{tocdepth}{1}
\vspace{-13pt}
\tableofcontents
\vspace{-13pt}

\section{Introduction}
The mirror symmetry of Calabi-Yau threefolds has occupied a very important place in the area of algebraic geometry. Roughly speaking, mirror symmetry is a conjecture that predicts the existence of mirror pairs $(M, W)$ of Calabi-Yau threefolds such that the complexified K\"ahler moduli space of $M$ is isomorphic to an open subset of the complex moduli space of $W$. This open subset is a neighborhood of a special boundary point known as the large complex structure limit \cite{CoxKatz,MarkGross}. The isomorphism between them is called mirror map, which is constructed by the identification of certain functions \cite{PhilipXenia,CoxKatz,MarkGross}. One function that plays a crucial role in mirror symmetry is prepotential, and it admits an instanton expansion with a constant term
\begin{equation}
Y_{000}=-3 \chi (M) \frac{\zeta(3)}{(2 \pi i)^3} +r,~ r\in \mathbb{Q},
\end{equation}
where $\chi(M)$ is the Euler characteristic of $M$. When the deformation of the mirror threefold $W$ forms a one-parameter algebraic family, the occurrence of $\zeta(3)$ in $Y_{000}$ has been carefully studied in the paper \cite{KimYang}, which also provides a motivic explanation to it. More precisely, the limit MHS on $H^3(W,\mathbb{Q})$ at the large complex structure limit splits into the direct sum
\begin{equation}
\mathbb{Q}(-1) \oplus \mathbb{Q}(-2) \oplus \mathbf{M},
\end{equation}
where $\mathbf{M}$ is an extension of $\mathbb{Q}(-3)$ by $\mathbb{Q}(0)$. The dual of $\mathbf{M}$ is an extension of $\mathbb{Q}(0)$ by $\mathbb{Q}(3)$, and its image in the group 
\begin{equation}
\text{Ext}^1_{\textbf{MHS}_{\mathbb{Q}}}\left(\mathbb{Q}(0),\mathbb{Q}(3)\right) \simeq  \mathbb{C}/(2 \pi i)^3\,\mathbb{Q}
\end{equation}
is the coset of a rational multiple of $(2 \pi i)^3Y_{000}$. The motivic nature of $\zeta(3)$ in $Y_{000}$ follows from the theory of mixed Tate motives and Ayoub's theory of motivic nearby cycle functor \cite{KimYang}. For a general mirror pair of Calabi-Yau threefolds with arbitrary Hodge number $h^{1,2}(W)$, it has been conjectured in the thesis \cite{Yangthesis} that the limit MHS on $H^3(W,\mathbb{Q})$ at the large complex structure limit splits into the direct sum
\begin{equation}
\mathbb{Q}(-1)^{h^{1,2}(W)} \oplus \mathbb{Q}(-2)^{h^{1,2}(W)} \oplus \mathbf{M}.
\end{equation}
The dual of $\mathbf{M}$ is an extension of $\mathbb{Q}(0)$ by $\mathbf{\mathbb{Q}}(3)$, whose image in $ \mathbb{C}/(2 \pi i)^3\,\mathbb{Q}$ is the coset of a rational multiple of $(2 \pi i)^3 Y_{000}$, thus it reveals the motivic nature of the $\zeta(3)$ in the mirror symmetry of Calabi-Yau threefolds. The Gamma conjecture claims that higher zeta values do appear in the theory of Calabi-Yau $n$-folds \cite{Galkin}, and it is natural to ask do they also have a motivic origin? 

The motivation of this paper is to study the Fermat pencil of Calabi-Yau $n$-folds, and try to see whether it can provide evidence to the Gamma conjecture, and moreover, whether these zeta values are motivic in nature. The Fermat pencil is a one parameter family of hypersurfaces in $\mathbb{P}^{n+1}$ defined by
\begin{equation}
\mathscr{X}_\psi :~\{~\sum_{i=0}^{n+1} X^{n+2}_i -(n+2)\, \psi\, \prod_{i=0}^{n+1} X_i =0~ \} \subset \mathbb{P}^{n+1},
\end{equation}
while from the adjunction formula, a smooth fiber $\mathscr{X}_\psi$ is a Calabi-Yau $n$-fold. The Picard-Fuchs equation of this family is an $(n+1)$-th order ODE \cite{Nagura}. The point $\psi=\infty$ is a special singular point that will be called the large complex structure limit, i.e. the monodromy of the solutions about $\infty$ is maximally unipotent. In a small neighborhood of it, the solution space of this Picard-Fuchs equation has a canonical basis that will be called the canonical periods. In order to construct the limit MHS at $\psi= \infty$, we will need to find a rational basis of the homology group $H_n(\mathscr{X}_\psi,\mathbb{Q})$, and compute the periods of the holomorphic $n$-form with respect to this basis, which will be called the rational periods \cite{KimYang,Schmid}. The linear transformation between the rational periods and the canonical periods is called the period matrix, which is a prerequisite to the construction of the limit MHS. However in general this period matrix is very difficult to compute, and in this paper, we will try a numerical method to evaluate the period matrix using Mathematica programs. Our numerical results have shown the occurrence of higher zeta values in the period matrix, which provides numerical evidence to the Gamma conjecture \cite{Galkin}.

The outline of this paper is as follows. In Section \ref{sec:fermatpencilcanonicalperiods}, we introduce the Fermat pencil of Calabi-Yau $n$-folds, and we will study its Picard-Fuchs equation by Frobenius method. We will prove a linear independence property about the solutions to this Picard-Fuchs equation. In Section \ref{sec:variationofHodgestructure}, we will discuss the variations of Hodge structures for the Fermat pencil, and study its limit MHS at the large complex structure limit, which is determined by a period matrix. In Section \ref{sec:numericalperiodmatrix}, we will use a numerical method to evaluate the period matrix for the cases where $n=4,5,6,7,8,9,10,11,12$ using Mathematica programs. Our numerical results have shown the appearance of higher zeta values in a highly interesting way. In Section \ref{sec:generalizationNdimensional}, we will generalize our numerical results in Section \ref{sec:numericalperiodmatrix}, and furthermore, we will give a motivic explanation to the appearance of higher zeta values in the period matrix.

\section{The Fermat pencil of Calabi-Yau \texorpdfstring{$n$}{n}-folds} \label{sec:fermatpencilcanonicalperiods}

In this section, we will introduce the Fermat pencil of Calabi-Yau $n$-folds. We will show how to solve its Picard-Fuchs equation by Frobenius method, which yields a canonical basis for the solution space. We will also prove a linear independence property for this canonical basis. 

\subsection{The Fermat pencil}

The Fermat pencil of Calabi-Yau $n$-folds is a one-parameter family of $n$-dimensional hypersurfaces in the projective space $\mathbb{P}^{n+1}$ defined by
\begin{equation} \label{eq:nplus2degreepolynomial}
\mathscr{X}_\psi :~\{~\sum_{i=0}^{n+1} X^{n+2}_i -(n+2)\, \psi\, \prod_{i=0}^{n+1} X_i =0~ \} \subset \mathbb{P}^{n+1},
\end{equation}
where $(X_0,X_1,\cdots,X_{n+1})$ forms the projective coordinate of $\mathbb{P}^{n+1}$. For convenience, the degree-$(n+2)$ polynomial in the formula \ref{eq:nplus2degreepolynomial} will be denoted by $f_\psi$. In a more formal language, the polynomial equation in the formula \ref{eq:nplus2degreepolynomial} defines a rational fibration
\begin{equation}
\pi: \mathscr{X} \rightarrow \mathbb{P}^1,
\end{equation}
whose singular fibers are over the points
\begin{equation}
\{\psi^{n+2}=1 \} \cup \{\psi= \infty \}.
\end{equation}
The adjunction formula tells us that $\mathscr{X}_\psi$ has vanishing first Chern class, so it is a Calabi-Yau manifold if it is smooth. Moreover, there exists a projective linear transformation
\begin{equation} \label{eq:projectivetransformationdescent}
X_0 \mapsto \zeta_{n+2} \,X_0, ~X_i \mapsto X_i,~i=1,\cdots,n+1;~\zeta_{n+2}=\exp 2 \pi i/(n+2)
\end{equation}
that induces an isomorphism between $\mathscr{X}_\psi$ and $\mathscr{X}_{\zeta_{n+2} \psi}$. Hence the `correct' parameter of the moduli is in fact $\psi^{n+2}$. For later convenience, let us instead define $\varphi$ by
\begin{equation}
\varphi:=\frac{1}{\psi^{n+2}}.
\end{equation}
The Fermat pencil in the formula \ref{eq:nplus2degreepolynomial} descends to a family over $\mathbb{P}^1$ which has $\varphi$ as its coordinate. By abuse of notation, we will denote this descent also by 
\begin{equation} \label{eq:oneparafamilyFermat}
\pi: \mathscr{X} \rightarrow \mathbb{P}^1,
\end{equation}
the singular fibers of which are over the points $\{ 0,1, \infty \}$.

\subsection{The Picard-Fuchs equation} \label{sec:picardfuchscanonicalperiods}

Given a smooth fiber $\mathscr{X}_\psi$, there is a canonical way to construct a holomorphic $n$-form on it \cite{MarkGross,Nagura}. On $\mathbb{P}^{n+1}$, there is a meromorphic  $(n+1)$-form $\Theta_\psi$ defined by
\begin{equation}
\Theta_\psi:=\sum_{i=0}^{n+1} \frac{(-1)^i\,\psi}{f_\psi}\left( X_i \,dX_0 \wedge \cdots \wedge \widehat{dX_i} \wedge \cdots \wedge d X_{n+1}\right),
\end{equation}
which is well-defined on the open subvariety $\mathbb{P}^{n+1}-\mathscr{X}_\psi$. Since $\Theta_\psi$ is a meromorphic $(n+1)$-form on $\mathbb{P}^{n+1}$, it is automatically closed. Therefore its residue along the hypersurface $\mathscr{X}_\psi$ is well-defined, which is by definition the holomorphic $n$-form $\Omega_\psi$ on $\mathscr{X}_\psi$
\begin{equation}
\Omega_\psi:= \text{Res}_{\mathscr{X}_\psi}(\Theta_\psi).
\end{equation}
The computation in the open affine subvarieties of $\mathscr{X}_\psi$ have shown that $\Omega_\psi$ is in fact nowhere vanishing \cite{MarkGross,Nagura}. Moreover, the meromorphic $(n+1)$-form $\Theta_\psi$ is equivariant under the projective transformation \ref{eq:projectivetransformationdescent}, hence its residue defines a nowhere vanishing holomorphic $n$-form $\Omega_\varphi$ on $\mathscr{X}_\varphi$. It is very important that $\Omega_\varphi$ has a logorithmic pole at the normal crossing divisors over the large complex structure limit $\varphi=0$ \cite{KimYang}.

For simplicity, the underlying differential manifold of a smooth fiber of the family \ref{eq:oneparafamilyFermat} will be denoted by $X$. A period of $\Omega_\varphi$ is by definition an integral of the form
\begin{equation}
\int_C \Omega_\varphi,~C \in H_n(X,\mathbb{C}),
\end{equation}
which is generally very difficult to compute directly. Fortunately, there is another much easier method to find the periods of $\Omega_\varphi$ by solving its Picard-Fuchs equation. It is well-known that the $n$-form $\Omega_\varphi$ satisfies the following Picard-Fuchs equation \cite{Nagura}
\begin{equation} \label{eq:picardfuchsfermatnfold}
\left( \vartheta^{n+1}-\varphi\, \prod_{k=1}^{n+1}\left( \vartheta+ \frac{k}{n+2} \right) \right)  \Omega_\varphi =0,~\vartheta =\varphi \frac{d}{d \varphi},
\end{equation}
which immediately implies that the periods of $\Omega_\varphi$ also satisfy this ODE. For simplicity, we will denote the Picard-Fuchs operator in the formula \ref{eq:picardfuchsfermatnfold} by
\begin{equation}
\mathcal{D}_n:=\vartheta^{n+1}-\varphi\, \prod_{k=1}^{n+1}\left( \vartheta+ \frac{k}{n+2} \right),~\vartheta =\varphi \frac{d}{d \varphi}.
\end{equation}

\subsection{The Frobenius method} \label{sec:frobeniusmethodandmonodromy}

The solutions to the Picard-Fuchs operator $\mathcal{D}_{n}$ can be found by the Frobenius method, which we now explain. First let $\varpi$ be formally defined by 
\begin{equation}
\varpi:=\varphi^{\epsilon} \, \sum_{k=0}^{\infty} a_k(\epsilon)\, \varphi^k,
\end{equation} 
where $\epsilon$ is a formal index and $a_k(\epsilon)$ is a number. For simplicity, we will denote the power series in $\varpi$ by $h(\epsilon,\varphi)$, i.e.
\begin{equation}
h(\epsilon,\varphi):=\sum_{k=0}^{\infty} a_k(\epsilon)\, \varphi^k.
\end{equation}
Now we plug $\varpi$ into the Picard-Fuchs equation
\begin{equation} \label{eq:solveingpicardfuchs}
\mathcal{D}_n\,\varpi=0,
\end{equation}
and we obtain equations of $\epsilon$ and $a_k(\epsilon)$. To the lowest order, we have
\begin{equation}
\epsilon^{n+1}=0,
\end{equation}
so $\varpi$ admits a series expansion in $\epsilon$ which terminates at $\epsilon^{n+1}$
\begin{equation}
\varpi=\sum_{i=0}^{n}\varpi_i \,\epsilon^i,~\text{with}~\varpi_i:=\left. \frac{\partial^i \varpi}{\partial \epsilon^i} \right|_{\epsilon=0}.
\end{equation}
Let us formally define $h_i(\varphi)$ by
\begin{equation}
h_i(\varphi):=\left. \frac{\partial^i h(\epsilon,\varphi)}{ \partial \epsilon^i} \right|_{\epsilon=0},
\end{equation}
which is a power series in $\varphi$, and we deduce that $\varpi_i$ is of the form
\begin{equation} \label{eq:formsofvpi}
\varpi_i=\left. \frac{\partial^i(\varphi^{\epsilon}\,h(\epsilon,\varphi))}{\partial \epsilon^i} \right|_{\epsilon=0}=\sum_{k=0}^{i} \binom{i}{k} h_k(\varphi)\,\log^{i-k} \varphi.
\end{equation}
The coefficients of the power series $h_i(\varphi)$ are determined by the recursion equations yielded by the Picard-Fuchs equation \ref{eq:solveingpicardfuchs}, which are unique if we impose the following boundary conditions
\begin{equation} \label{eq:boundaryconditionscanonicalperiods}
h_0(0)=1,~h_1(0)=\cdots=h_n(0)=0.
\end{equation}
Since the Picard-Fuchs operator $\mathcal{D}_n$ lies in the (non-commutative) ring $\mathbb{Q}[\varphi,\vartheta]$, the coefficients of the recursion equations lie in $\mathbb{Q}$, from which we obtain
\begin{equation}
h_i(\varphi) \in \mathbb{Q}[[\varphi]],~i=0,1,\cdots,n.
\end{equation}
The Picard-Fuchs operator $\mathcal{D}_n$ has three regular singularities $\{0,~1,~\infty \}$, therefore the power series expansion of $h_i(\varphi)$ converges on the unit disc 
\begin{equation}
\Delta:=\{ |\varphi| <1 \}.
\end{equation}
By analytic continuation, the solution $\varpi_i$ extends to a multi-valued holomorphic function on $\mathbb{C}-\{0,1 \}$, which will be called the canonical period of $\mathcal{D}_n$. The canonical period vector is by definition the column vector
\begin{equation}
\varpi:=\left( \varpi_0,\varpi_1,\cdots,\varpi_n \right)^\top.
\end{equation}
The normalized canonical period vector $\varpi_R$ is defined to be the column vector 
\begin{equation}
\varpi_R:=\left( \varpi_0,\frac{1}{2 \pi i}\varpi_1,\cdots,\frac{1}{(2 \pi i)^n}\varpi_n \right)^\top,~\varpi_{R,j}=\frac{1}{(2 \pi i)^j}\varpi_j .
\end{equation}
The monodromy of the (normalized) canonical periods about $\varphi=0$ is induced by the analytic continuation $\log \varphi \rightarrow \log \varphi+ 2\pi i$, under which $\varpi_R$ transforms in the way
\begin{equation} \label{eq:rescalemonodromyperiods}
T_0:\varpi_{R,i}(\varphi) \mapsto \sum_{k=0}^i \binom{i}{k} \varpi_{R,k}(\varphi).
\end{equation}
More explicitly, $T_0$ is the matrix
\begin{equation} \label{eq:monodromymatrixT0varpir}
T_0=
\begin{pmatrix}
1, & 0, & 0, & 0,  &  \cdots &0, \\
1, & 1, & 0,  &  0, &  \cdots & 0, \\
1, & 2, & 1, & 0,   &   \cdots &0,\\
\vdots & \vdots &\vdots & \vdots &  \ddots & \vdots \\
1, & \binom{n}{1}, & \binom{n}{2}, & \binom{n}{3},  & \cdots & \binom{n}{n}, \\
\end{pmatrix},
\end{equation}
and the monodromy action can be written as
\begin{equation}
\varpi_R \rightarrow T_0 \,\varpi_R.
\end{equation}
The matrix $T_0$ satisfies the equation
\begin{equation}
(T_0-\text{Id})^{n+1}=0,
\end{equation}
therefore the monodromy around 0 is maximally unipotent, and we will call $\varphi=0$ the large complex structure limit.

\subsection{The Wronskian of canonical periods} \label{sec:linearlyindependence}

Given a point $\varphi \in \mathbb{C}-\{ 0,1 \}$, we will prove the following $n+1$ vectors
\begin{equation} \label{eq:linearindependencephi0}
\varpi(\varphi), \vartheta \varpi(\varphi),\cdots, \vartheta^n \varpi(\varphi)
\end{equation}
are linearly independent at $\varphi$, a property that will be important in later sections. First, we will need to introduce the Wronskian of canonical periods
\begin{equation}
U_{ij}(\varphi):=\frac{1}{i!} \,\vartheta^i \, \varpi_j(\varphi),
\end{equation}
where the indices $i,j$ run from 0 to $n$. The determinant of the Wronskian $U(\varphi)$, denoted by $\text{det}\left( U(\varphi) \right)$, is invariant under the monodromy around $0$, hence it must be a single-valued holomorphic function on the unit disc $\Delta$. Therefore when we compute $\text{det}\left( U(\varphi) \right)$, we can ignore the terms of $U_{ij}(\varphi)$ that depend on $\log \varphi$. From the boundary condition \ref{eq:boundaryconditionscanonicalperiods}, we find
\begin{equation} \label{eq:determinantboundarycondition}
\lim_{\varphi \rightarrow 0} \text{det}\left( U(\varphi) \right)=1.
\end{equation}
From elementary properties of the determinant of a matrix and the Picard-Fuchs equation \ref{eq:picardfuchsfermatnfold}, $\text{det}\left( U(\varphi) \right)$ satisfies a first order differential equation
\begin{equation}
\left(1- \varphi \right) \vartheta \left( \text{det}\left( U(\varphi) \right) \right)=\frac{1}{2}\,(n+1)\, \varphi \,\text{det}\left( U(\varphi) \right).
\end{equation}
Together with the boundary condition \ref{eq:determinantboundarycondition}, we immediately get
\begin{equation}
\text{det}\left( U(\varphi) \right)=\left(1- \varphi \right)^{-\frac{n+1}{2}},
\end{equation}
which implies the canonical periods $\{\varpi_i \}_{i=0}^n$ are linearly independent and form a basis for the solution space of $\mathcal{D}_n$. If $\varphi \in \mathbb{C}-\{0,1\}$, we have
\begin{equation}
\text{det}\left( U(\varphi) \right) \neq 0,
\end{equation}
hence the vectors in the formula \ref{eq:linearindependencephi0} must be linearly independent.

\section{Variation of Hodge structures} \label{sec:variationofHodgestructure}

In this section, we will discuss the variation of Hodge structures for the Fermat pencil of Calabi-Yau $n$-folds, and we will study the limit MHS on $H^n(X,\mathbb{Q})$ at the large complex structure limit $\varphi=0$.

\subsection{The split of pure Hodge structures}

Recall that $X$ denotes the underlying differential manifold of an arbitrary smooth fiber of the family \ref{eq:oneparafamilyFermat}. Let $H^a_n (X, \mathbb{Q})$ be the subspace of $H_n(X,\mathbb{Q})$ defined by the property
\begin{equation}
A \in H^a_n (X, \mathbb{Q}) \iff \int_A \Omega_\varphi \equiv 0.
\end{equation}
The Poincar\'e duality induces a non-degenerate bilinear form on $H_n(X,\mathbb{Q})$, and let $H^t_n (X, \mathbb{Q})$ be the complement of $H^a_n (X, \mathbb{Q})$ with respect to it, thus we have
\begin{equation} \label{eq:homologydecomposition}
H_n(X, \mathbb{Q})=H^a_n (X, \mathbb{Q}) \oplus H^t_n (X, \mathbb{Q}).
\end{equation}
Let the dual of $H^t_n (X, \mathbb{Q})$ (resp. $H^a_n (X, \mathbb{Q})$) be denoted by $H^{n,t} (X, \mathbb{Q})$ (resp.  $H^{n,a} (X, \mathbb{Q})$), then the cohomology group $H^n(X,\mathbb{Q})$ splits into the direct sum
\begin{equation} \label{eq:cohomologydecomposition}
H^n (X, \mathbb{Q})=H^{n,a} (X, \mathbb{Q}) \oplus H^{n,t} (X, \mathbb{Q}).
\end{equation}
From the definition of $H^a_n (X, \mathbb{Q})$, the monodromy action around $\varphi=0$ preserves the decomposition of $H_n(X,\mathbb{Q})$ (resp. $H^n(X,\mathbb{Q})$) in the formula \ref{eq:homologydecomposition} (resp. \ref{eq:cohomologydecomposition}). Let us denote the monodromy action on $H^t_n(X,\mathbb{Q})$ by $T$, while the operator $N$ is defined by
\begin{equation}
N:=T-\text{Id},
\end{equation}
furthermore, $T$ and $N$ extend to maps on $H^t_n (X, \mathbb{C}):=H^t_n (X, \mathbb{Q})  \otimes \mathbb{C}$.

The (nontrivial) period of $\Omega_\varphi$ is given by its integration over the cycles in $H^t_n (X, \mathbb{C}) $. From Section \ref{sec:fermatpencilcanonicalperiods}, the $n$-form $\Omega_\varphi$ has $n+1$ linearly independent canonical periods $\{ \varpi_i \}_{i=0}^n$, hence the dimension of $H^t_n (X, \mathbb{C}) $ is $n+1$. Moreover, there exists a basis $\{C_i \}_{i=0}^n$ of the complex vector space $H^t_n (X, \mathbb{C})$ such that 
\begin{equation}
\varpi_i(\varphi)=\int_{C_i} \Omega_\varphi,
\end{equation}
from which we have
\begin{equation}
\varpi_{R,j}(\varphi)=\int_{C_{R,j}} \Omega_\varphi,~C_{R,j}=\frac{1}{(2 \pi i)^j} C_{R,j}.
\end{equation}
The monodromy action on the basis $C_{R,i}$ is determined by the monodromy action on $\varpi_{R,i}$ in the formula  \ref{eq:rescalemonodromyperiods}
\begin{equation} \label{eq:homologyrationalmonodromy}
T:C_{R,i} \mapsto \sum_{k=0}^i \binom{i}{k} C_{R,k},
\end{equation}
therefore the operator $N$ satisfies
\begin{equation} \label{eq:nilpotentofN}
N^{n+1}=0.
\end{equation}

 Let the dual of the basis $\{C_i \}_{i=0}^n$ be denoted by $\{\gamma_i \}_{i=0}^n$, i.e they satisfy the pairing relation
\begin{equation}
\gamma_i(C_j)=\delta_{ij},
\end{equation}
and $\{\gamma_i \}_{i=0}^n$ forms a basis of $H^{n,t} (X, \mathbb{C}):=H^{n,t} (X, \mathbb{Q}) \otimes \mathbb{C}$.  The $n$-form $\Omega_\varphi$ admits an expansion of the form
\begin{equation}
\Omega_\varphi =\sum_{i=0}^{n} \gamma_i \,\int_{C_i} \Omega_\varphi =\sum_{i=0}^n \gamma_i \,\varpi_i(\varphi),
\end{equation}
and similarly, the form $\vartheta^k \Omega(\varphi)$ admits an expansion
\begin{equation}
\vartheta^k \Omega(\varphi)=\sum_{i=0}^n \gamma_i \int_{C_i} \vartheta^k \Omega(\varphi)=\sum_{i=0}^n \gamma_i\, \vartheta^k\varpi_i(\varphi).
\end{equation}
While the result from Section \ref{sec:linearlyindependence} implies that for every $\varphi \in \mathbb{C} -\{ 0, 1\}$, the forms
\begin{equation} \label{eq:linearindependencethetaomega}
\Omega_\varphi,~ \vartheta \Omega_\varphi, ~\cdots, ~\vartheta^n \Omega_\varphi
\end{equation}
are linearly independent, therefore they form a basis of $H^{n,t}(X,\mathbb{C})$. 

From Hodge theory, there exists a Hodge decomposition 
\begin{equation}
H^n(X, \mathbb{Q}) \otimes \mathbb{C}=H^{n,0}(\mathscr{X}_\varphi) \oplus H^{n-1,1}(\mathscr{X}_\varphi) \oplus \cdots \oplus H^{1,n-1}(\mathscr{X}_\varphi) \oplus H^{0,n}(\mathscr{X}_\varphi).
\end{equation}
It defines a weight-$n$ pure Hodge structure $\left( H^n(X, \mathbb{Q}), F^p_\varphi \right)$ with the Hodge filtration $F_\varphi^p$ 
\begin{equation}
F^p_\varphi= \oplus_{k \geq p}H^{k,n-k}(\mathscr{X}_\varphi),
\end{equation}
which varies holomorphically with respect to $\varphi$. From Griffiths transversality, we have 
\begin{equation}
\vartheta^i \Omega_\varphi \in F^{n-i}_{\varphi},~i=0,1,\cdots,n.
\end{equation}
Together with the linear independence of $\vartheta^i \Omega_\varphi$, it shows there is a pure Hodge structure on $H^{n,t}(X,\mathbb{Q})$ with Hodge filtration given by \cite{PetersSteenbrink}
\begin{equation}
F^{p,t}_\varphi :=\oplus_{i=0}^{n-p}~ \mathbb{C} ~\vartheta^i \Omega(\varphi).
\end{equation}
Therefore the pure Hodge structure on $H^n(X,\mathbb{Q})$ splits into the direct sum
\begin{equation} \label{eq:decompositionpureHodgestructure}
\left( H^n(X,\mathbb{Q}),F^p_\varphi \right) =\left(H^{n,a} (X, \mathbb{Q}),F_\varphi^{p,a} \right) \oplus \left( H^{n,t} (X, \mathbb{Q}),F_\varphi^{p,t}\right),
\end{equation}
where the pure Hodge structure $\left(H^{n,a} (X, \mathbb{Q}),F_\varphi^{p,a} \right)$ is induced by $\left( H^n (X, \mathbb{Q}), F^p_\varphi\right)$. Moreover, this decomposition is preserved by the monodromy action about $\varphi=0$.

\subsection{The limit MHS} \label{sec:limitmhs}

Now we will need the theory of limit MHS \cite{Schmid, SteenbrinkLMHS}. When $\varphi$ approaches 0, the pure Hodge structure $\left( H^n(X,\mathbb{Q}),F^p_\varphi \right)$ has a limit that is a MHS
\begin{equation} \label{eq:limitMHSofFermatpencil}
\left(H^n(X,\mathbb{Q}),W_q, F^p_{\text{lim}} \right),
\end{equation}
which is called the limit MHS. Here $W_q$ is an increasing weight filtration on $H^n(X,\mathbb{Q})$ that is determined by the monodromy action on $H^n(X,\mathbb{Q})$, while $F^p_{\text{lim}}$ is the limit Hodge filtration on $H^n(X,\mathbb{C})$ determined by the `limit' of the Hodge filtration $F^p_\varphi$ \cite{MarkGross,Schmid, SteenbrinkLMHS}. Because the monodromy action preserves the split in the formula \ref{eq:decompositionpureHodgestructure}, we immediately deduce that the limit MHS in the formula \ref{eq:limitMHSofFermatpencil} also splits into the direct sum
\begin{equation}\label{eq:splitsoflmhs}
\left(H^n(X,\mathbb{Q}),W_q, F^p_{\text{lim}} \right)=\left(H^{n,a}(X,\mathbb{Q}),W^a_q, F^{p,a}_{\text{lim}} \right) \oplus \left(H^{n,t}(X,\mathbb{Q}),W^t_q, F^{p,t}_{\text{lim}} \right).
\end{equation}
The limit Hodge filtration $F^{p,t}_{\text{lim}} $ on $H^{n,t}(X,\mathbb{C})$ is given by a regularized limit of the vectors $\{ \vartheta^i\varpi \}_{i=0}^n$, which can be explicitly computed \cite{KimYang,PetersSteenbrink,Schmid}. The weight filtration $W^t_q$ on $H^{n,t}(X,\mathbb{Q})$ is determined by the monodromy action on $H^{n,t}(X,\mathbb{Q})$ \cite{PetersSteenbrink}
\begin{equation}
T^*:H^{n,t}(X,\mathbb{Q}) \rightarrow H^{n,t}(X,\mathbb{Q}).
\end{equation}
Here $T^*$ is the dual of $T$. If we define the operator $N^*$ by 
\begin{equation}
N^*:=T^* -\text{Id},
\end{equation}
it also satisfies
\begin{equation}
(N^*)^{n+1}=0.
\end{equation}
Since the dimension of $H^{n,t}(X,\mathbb{Q})$ is $n+1$, the semi-simplification of $\left(H^{n,t}(X,\mathbb{Q}),W^t_q, F^{p,t}_{\text{lim}} \right)$ must be
\begin{equation} \label{eq:semisimplificationLMHS}
\mathbb{Q}(0) \oplus \mathbb{Q}(-1) \oplus \cdots \oplus \mathbb{Q}(-n),
\end{equation}
thus we deduce that it is a mixed Hodge-Tate structure  \cite{KimYang,PetersSteenbrink}. But in order to construct this MHS explicitly, we will need to know the periods of $\Omega_\varphi$ with respect to a rational basis of $H^{n,t}(X,\mathbb{Q})$. Or equivalently, the linear transformation between $\{C_{R,i} \}_{i=0}^n$ (a basis of $H_n^t(X,\mathbb{C})$) and a rational basis of $H_n^t(X,\mathbb{Q})$, while the matrix of this linear transformation will be called the period matrix.

\subsection{The period matrix} \label{sec:theperiodsmatrixanalysis}

Let us first construct a special basis of the rational vector space $ H_n^t(X,\mathbb{Q})$, which can simplify the period matrix significantly. Since the operator $N$ satisfies the formula \ref{eq:nilpotentofN}, the kernel $\text{ker}\,N$ ($=\text{Im}\,N^n$) is a one dimensional subspace of $ H_n^t(X,\mathbb{Q})$ spanned by an element $A_0$. Moreover, there exists an element $A_1\in \text{Im}\,N^{n-1} $ such that
\begin{equation}
N(A_1) =A_0.
\end{equation}
An easy induction shows there exist elements $\{ A_i \}_{i=2}^n$ that satisfy
\begin{equation} \label{eq:conditionmonodromyAi}
N(A_i)= \sum_{k=0}^{i-1} \binom{i}{k} A_k,
\end{equation}
and furthermore, $\{A_i\}_{i=0}^n$ forms a basis of $ H_n^t(X,\mathbb{Q})$. The choice of $A_i$ is not unique, and we have the freedom to redefine $A_n$ by
\begin{equation} \label{eq:rationaltransformationAi}
A'_n=A_n+\sum_{i=0}^{n-1} l_i\, A_i,~l_i\in \mathbb{Q},
\end{equation}
which will yield a new basis $\{ A'_i \}_{i=0}^n$ that also satisfies the formula \ref{eq:conditionmonodromyAi}.  Since the dimension of $\text{ker}\,N$ is one, the element $C_{R,0}$ must be a nonzero multiple of $A_0$, hence after a rescaling of the $n$-form $\Omega_\varphi$ we can assume 
\begin{equation}
C_{R,0}=A_0.
\end{equation}
From the formula \ref{eq:homologyrationalmonodromy}, we now must have
\begin{equation}
N(A_1-C_{R_1})=0, 
\end{equation}
which implies
\begin{equation}
A_1=C_{R,1}+a_{1,0}\,C_{R,0},~a_{1,0} \in \mathbb{C}.
\end{equation}
A simple induction shows
\begin{equation} \label{eq:rationalcanonicaltransformations}
A_i=C_{R,i}+\sum_{k=0}^{i-1} a_{i,k}\, C_{R,k},~i=2,\cdots,n, ~a_{i,k} \in \mathbb{C}.
\end{equation}
A different choice of the basis $\{A_i \}_{i=0}^n$ induced by the formula \ref{eq:rationaltransformationAi} will change the value of $a_{n,i}$ to
\begin{equation}
a_{n,i} \rightarrow a_{n,i}+r_{n,i},~r_{n,i} \in \mathbb{Q}.
\end{equation}
While at the same time, the values of $a_{k,i},k<n$ will change to
\begin{equation}
a_{k,i} \rightarrow a_{k,i}+r_{k,i}, r_{k,i} \in \mathbb{Q},
\end{equation}
where $\{ r_{k,i} \}$ is determined by $\{ r_{n,i} \}$. In particular, we deduce that the coset of $a_{i,j}$ in $\mathbb{C}/\mathbb{Q}$ does not depend on the choice of the rational basis $\{A_i \}_{i=0}^n$.

The rational period $\Pi_i$ of $\Omega_\varphi$ is defined by the integration
\begin{equation} \label{eq:defnrationalperiods}
\Pi_i:=\int_{A_i} \,\Omega_\varphi,
\end{equation}
and let us define the rational period vector $\Pi$ to be the column vector
\begin{equation}
\Pi:=\left(\Pi_0,\cdots,\Pi_n \right)^\top.
\end{equation}
From the formula \ref{eq:rationalcanonicaltransformations}, the linear transformation between the rational period $\Pi_i$ and the canonical period $\varpi_i$ is given by
\begin{equation}
\Pi_i=\varpi_{R,i}+\sum_{k=0}^{i-1} a_{i,k}\,\varpi_{R,k}.
\end{equation}
It is more convenient to write the period matrix $P$ as
\begin{equation} \label{eq:periodmatrixgeneral}
P=
\begin{pmatrix}
1, & 0, & 0,   & \cdots &0, \\
a_{1,0}, & 1, &  0, & \cdots & 0, \\
a_{2,0}, & a_{2,1}, & 1,  & \cdots &0,\\
\vdots & \vdots & \vdots & \ddots & \vdots \\
a_{n,0}, & a_{n,1}, & a_{n,2}, & \cdots & 1, \\
\end{pmatrix},
\end{equation}
and we have
\begin{equation}
\Pi=P \cdot  \varpi_{R}.
\end{equation}
The monodromy matrix of $\Pi$ is given by
\begin{equation}
P \cdot T_0 \cdot P^{-1}.
\end{equation}
But from the definition of $\Pi$, its monodromy matrix must lie in $\text{GL}(n+1,\mathbb{Q})$, and this imposes strong restrictions on the value of $a_{i,j}$.
\begin{lemma}
After a change of $a_{n,i}$ to $a_{n,i}+r_{n,i}$ for some rational number $r_{n,i}$, we have
\begin{equation} \label{eq:conditionsonaijvalues}
a_{i,j}=\binom{i}{j} \, a_{i-j,0},~\forall i>j.
\end{equation}
\end{lemma}
\begin{proof}
This follows from a direct computation of the monodromy matrix $P \cdot T_0 \cdot P^{-1}$ of $\Pi$ from the monodromy matrix $T_0$ of $\varpi_{R,i}$ about $\varphi=0$.
\end{proof}
From now on we will assume that we have adjusted the values of $a_{n,i}$ by a rational number and the formula \ref{eq:conditionsonaijvalues} is satisfied. Once the period matrix $P$ is known, the limit MHS $\left(H^{n,t}(X,\mathbb{Q}),W^t_q, F^{p,t}_{\text{lim}} \right)$ can be constructed explicitly. The readers are referred to the paper \cite{KimYang} for the details of the computations in the case of Calabi-Yau threefolds, which can be generalized to higher dimensional cases. However, in general it is extremely difficult to compute the period matrix $P$, thus in this paper, we will resort to a numerical approach.

\section{The numerical evaluation of the period matrix} \label{sec:numericalperiodmatrix}

In this section, we will use a numerical method to compute the period matrix when the dimension $n$ is $4,5,6,7,8,9,10,11,12$. More precisely, we will compute the first 100 digits of the entries of the period matrix $P$, and show they agree with zeta values. Let us first briefly explain the idea of this numerical method, and illustrate how it works in the case where $n=1$, i.e. the Hesse pencil of elliptic curves.

\subsection{The period matrix of the Hesse pencil}

It is important to notice that the Picard-Fuchs operator $\mathcal{D}_n$ has three regular singularities $\{0,1,\infty \}$. Suppose the monodromy matrix of the period vector $\varpi_R$ around 1 is given by $T_1$, i.e.
\begin{equation}
\varpi_R \mapsto T_1 \,\varpi_R.
\end{equation}
Then the monodromy matrix of the integral period vector $\Pi$ is given by $P \cdot T_1 \cdot P^{-1}$, which is a matrix with entries being rational numbers. We do not know how to compute $T_1$, because we do not know how to analytically extend $\varpi_{R,i}$ to a small neighborhood of $\varphi=1$ explicitly. But the good news is that $T_1$ can be numerically evaluated. The expansion of $\varpi_{R,i}$ obtained in the Section \ref{sec:frobeniusmethodandmonodromy} converges in the unit disc $|\varphi|<1$. Use Mathematica, we can compute the expansion of $h_i(\varphi)$ to a high order. Then we choose a point, say $\varphi_0=1/10$, and the expansion of $\varpi_{R,i}$ gives us a very precise value of $\varpi_{R,i}(1/10)$ numerically. Let us compute the first 200 digits of $\varpi_{R,i}(1/10)$. Then we can numerically solve the Picard-Fuchs equation on a contour about 1, from which we can numerically compute the monodromy matrix $T_1$ (to 100 digits). We now plug the numerical value of $T_1$ into $P \cdot T_1 \cdot P^{-1}$, and the property that its entries are rational numbers impose strong restrictions on the values of the period matrix $P$.

We now explain how this method works for the toy example where $n=1$. The Picard-Fuchs operator $D_1$ is of the form
\begin{equation}
\mathcal{D}_1:=\vartheta^2-\varphi \,\prod_{k=1}^{2}\left(\vartheta+ \frac{k}{3} \right),
\end{equation}
which is a second order ODE. The Frobenius method yields two canonical solutions $\{\varpi_i \}_{i=0}^1$ of the form shown in Section \ref{sec:frobeniusmethodandmonodromy}. The power series expansions of $\{ h_i \}_{i=0}^1$ can be computed by Mathematica program very efficiently, and we have computed the first 10\,000 terms
\begin{equation} \label{eq:hesseh0h1}
\begin{aligned}
h_0&=1 +\frac{2\, \varphi }{9}+\frac{10 \,\varphi ^2}{81}+\frac{560\, \varphi ^3}{6561}+\frac{3850 \,\varphi ^4}{59049}+\frac{28028 \,\varphi ^5}{531441}+\cdots, \\
h_1 &=\frac{5 \,\varphi }{9}+\frac{19\, \varphi ^2}{54}+\frac{5018 \,\varphi ^3}{19683}+\frac{141355\, \varphi ^4}{708588}+\frac{522109 \, \varphi ^5}{3188646}+ \cdots. \\
\end{aligned}
\end{equation}
We can use their expansions to numerically compute the values of $\{\varpi_{R,0},\varpi_{R,1} \}$ at the following two points
\begin{equation}
\varphi_0=\frac{3}{20}+\frac{1}{10}i, ~\varphi_1=\frac{1}{10}+\frac{1}{10}i.
\end{equation}
Next, we choose a contour around the singular point $\varphi=1$, and numerically solve the Picard-Fuchs equation, from which we obtain the numerical evaluation of the monodromy matrix $T_1$. We have computed the first 100 digits of $T_1$, and here we only list the first 20 digits
\begin{equation}
T_1-\text{Id}=
\begin{pmatrix}
1.5736461865472690010\, i, & 3,  \\
0.82545410681158738285, & -1.5736461865472690010\, i, \\
\end{pmatrix}.
\end{equation}
From Section \ref{sec:theperiodsmatrixanalysis}, there exists a rational basis such that the period matrix $P$ is of the form
\begin{equation}
P=
\begin{pmatrix}
1, & 0, \\
a_{1,0},& 1,\\
\end{pmatrix}.
\end{equation}
The numerical value of the $(1,1)$ entry of $P \cdot (T_1-\text{Id}) \cdot P^{-1} $ is given by
\begin{equation}
1.5736461865472690010\, i-3 \, a_{1,0}.
\end{equation}
Since this entry must be a rational number, after a change of basis we can let $a_{1,0}$ be
\begin{equation}
a_{1,0}=0.5245487288490896670\,i,
\end{equation}
which agrees with the number ${-3 \log 3}/{2 \pi i}$. On the other hand, if we let $a_{1,0}$ be ${-3 \log 3}/{2 \pi i}$, our numerical computations show that up to the first 100 digits, we have
\begin{equation}
 P \cdot T_1 \cdot P^{-1}=
\begin{pmatrix}
1, & 3, \\
0,& 1,\\
\end{pmatrix}.
\end{equation}
The method in this section generalize to the cases where $n \geq 4$.

\subsection{The period matrix of sextic Calabi-Yau fourfold}

We will now look at the case where $n=4$, i.e. the Fermat pencil of sextic Calabi-Yau fourfold. The Picard-Fuchs operator $D_4$ is of the form
\begin{equation}
\mathcal{D}_4:=\vartheta^5-\varphi \,\prod_{k=1}^{5}\left(\vartheta+ \frac{k}{6} \right),
\end{equation}
which is a fifth order ODE. The Frobenius method yields five canonical solutions $\{\varpi_i \}_{i=0}^4$ of the form shown in Section \ref{sec:frobeniusmethodandmonodromy}, i.e. formula \ref{eq:formsofvpi}. The power series expansions of $\{ h_i \}_{i=0}^4$ can be computed by Mathematica program very efficiently, and here we give the first several terms
\begin{equation}
\begin{aligned}
h_0&=1 +\frac{5\, \varphi }{324}+\frac{1925 \,\varphi ^2}{559872}+\frac{14889875\, \varphi ^3}{11019960576}+\cdots, \\
h_1 &=\frac{29\, \varphi }{216}+\frac{222205\, \varphi ^2}{6718464}+\frac{1187693675\, \varphi ^3}{88159684608}+ \cdots, \\
h_2 &=\frac{35\, \varphi }{81}+\frac{63079\, \varphi ^2}{419904}+\frac{107868232835\, \varphi ^3}{1586874322944}+\cdots, \\
h_3 &=-\frac{35 \, \varphi }{27}-\frac{7175 \, \varphi ^2}{139968}+\frac{14522912269\, \varphi ^3}{396718580736}+\cdots, \\
h_4 &=-\frac{5495 \, \varphi ^2}{4374}-\frac{25808545\, \varphi ^3}{45349632}+\cdots. \\
\end{aligned}
\end{equation}
We have computed the first 500 terms of the series expansion of $h_i(\varphi)$ by Mathematica program. Using the numerical method explained above, we have numerically evaluated the monodromy matrix $T_1$ of $\varpi_R$ about 1 to the first 100 digits. For convenience, let us define
\begin{equation}
f_{\log}(n):=-\frac{n \log n}{2 \pi i},~n \in \mathbb{Z}_+.
\end{equation}
We now impose the condition that $P\cdot T_1 \cdot P^{-1}$ is a matrix with entries being rational numbers, and we find that this condition determines the period matrix $P$ uniquely up to an isomorphism. Numerically we have checked that up to 100 digits, the period matrix $P$ is of the form
\begin{equation}
P=P_{\zeta} \cdot P_{\log},
\end{equation}
where the matrix $P_{\zeta}$ is given by
\begin{equation} \label{eq:sexticfrobenius}
P_{\zeta}=
\begin{pmatrix}
1, & 0, & 0, & 0, & 0, \\
0, & 1, & 0, & 0, & 0, \\
0, & 0, & 1, & 0, & 0, \\
 -420\, \zeta(3)/(2 \pi i)^3,  & 0, & 0, & 1, & 0, \\
0, &  -1\,680 \,\zeta(3)/(2 \pi i)^3, & 0, & 0, & 1, \\
\end{pmatrix},
\end{equation}
and the matrix $P_{\log}$ is given by
\begin{equation}
P_{\log}=
\begin{pmatrix}
1, & 0, & 0, & 0, & 0, \\
f_{\log}(6), & 1, & 0, & 0, & 0, \\
f_{\log}(6)^2, &2 f_{\log}(6), & 1, & 0, & 0, \\
f_{\log}(6)^3, & 3 f_{\log}(6)^2, & 3f_{\log}(6), & 1, & 0, \\
f_{\log}(6)^4, & 4 f_{\log}(6)^3, & 6 f_{\log}(6)^2, & 4f_{\log}(6), & 1, \\
\end{pmatrix}.
\end{equation}
We define the normalized canonical period $\varpi_{M,j}$ by
\begin{equation}
\varpi_{M,j}:=\frac{1}{(2 \pi i)^j} \, \sum_{k=0}^{j} \binom{j}{k} h_k(\varphi)\,\log^{j-k} \left( 6^{-6} \varphi \right),
\end{equation}
which is still a solution to the Picard-Fuchs operator $\mathcal{D}_4$. Let the normalized canonical period vector $\varpi_M$ be
\begin{equation}
\varpi_{M}:=\left(\varpi_{M,0},\cdots, \varpi_{M,4} \right)^\top,
\end{equation}
then the linear transformation between $\Pi$ and $\varpi_M$ is given by
\begin{equation}
\Pi=P_{\zeta} \cdot \varpi_M.
\end{equation}
In particular, the periods $\varpi_{M,0}$ and $\varpi_{M,1}$ are given by the integration of the 4-form $\Omega_\varphi$ over two rational cycles of $H_4(X,\mathbb{Q})$, which plays an important role in the study of the mirror symmetry of sextic fourfold.

\subsection{Further examples}

Now we will list our computations for the cases where $n=5,6,7,8,9,10,11,12$.

\subsubsection{Septic Calabi-Yau 5-folds}

For septic Calabi-Yau 5-folds, the same method has shown that the period matrix $P$ is of the form
\begin{equation}
P=P_{\zeta} \cdot P_{\log}.
\end{equation}
If we define the normalized canonical period $\varpi_{M,j}$ by
\begin{equation}
\varpi_{M,j}:=\frac{1}{(2 \pi i)^j} \, \sum_{k=0}^{j} \binom{j}{k} h_k(\varphi)\,\log^{j-k} \left( 7^{-7} \varphi \right),
\end{equation}
then the linear transformation between $\varpi_M$ and $\varpi_R$ is given by
\begin{equation}
\varpi_M=P_{\log} \cdot \varpi_R.
\end{equation}
While the linear transformation between $\Pi$ and $\varpi_M$ is given by
\begin{equation}
\Pi=P_{\zeta} \cdot \varpi_M.
\end{equation}
The entries of the matrix $P_{\zeta}$ satisfy
\begin{equation}
(P_{\zeta})_{i,i}=1;~(P_{\zeta})_{i,j}=0,~\forall j>i;~(P_{\zeta})_{i,j}=\binom{i}{j}(P_{\zeta})_{i-j,0},~\forall j<i;
\end{equation}
where the indices $i$ and $j$ run from 0 to 5. Now let $\tau_{5,3}$ and $\tau_{5,5}$ be
\begin{equation}
\tau_{5,3}=-112\, \zeta(3)/(2 \pi i)^3,~\tau_{5,5}=-3\,360 \, \zeta(5)/(2 \pi i)^5.
\end{equation} 
Our numerical method has shown that
\begin{equation}
(P_{\zeta})_{1,0}=(P_{\zeta})_{2,0}=(P_{\zeta})_{4,0}=0,~(P_{\zeta})_{3,0}=3!\, \tau_{5,3},(P_{\zeta})_{5,0}=5! \,\tau_{5,5}.
\end{equation}

\subsubsection{Octic Calabi-Yau 6-folds}

For octic Calabi-Yau 6-folds, the same method has shown that the period matrix $P$ is also of the form
\begin{equation}
P=P_{\zeta} \cdot P_{\log}.
\end{equation}
If we define the normalized canonical period $\varpi_{M,j}$ by
\begin{equation}
\varpi_{M,j}:=\frac{1}{(2 \pi i)^j} \, \sum_{k=0}^{j} \binom{j}{k} h_k(\varphi)\,\log^{j-k} \left( 8^{-8} \varphi \right),
\end{equation}
then the linear transformation between $\varpi_M$ and $\varpi_R$ is given by
\begin{equation}
\varpi_M=P_{\log} \cdot \varpi_R.
\end{equation}
While the linear transformation between $\Pi$ and $\varpi_M$ is given by
\begin{equation}
\Pi=P_{\zeta} \cdot \varpi_M.
\end{equation}
The entries of the matrix $P_{\zeta}$ satisfy
\begin{equation}
(P_{\zeta})_{i,i}=1;~(P_{\zeta})_{i,j}=0,~\forall j>i;~(P_{\zeta})_{i,j}=\binom{i}{j}(P_{\zeta})_{i-j,0},~\forall j<i;
\end{equation}
where the indices $i$ and $j$ run from 0 to 6.  Now let $\tau_{6,3}$ and $\tau_{6,5}$ be
\begin{equation}
\tau_{6,3}=-168\, \zeta(3)/(2 \pi i)^3,~\tau_{6,5}=-6\,552 \, \zeta(5)/(2 \pi i)^5.
\end{equation}
Our numerical results have shown that
\begin{equation}
(P_{\zeta})_{1,0}=(P_{\zeta})_{2,0}=(P_{\zeta})_{4,0}=0,~(P_{\zeta})_{3,0}=3!\, \tau_{6,3},(P_{\zeta})_{5,0}=5! \,\tau_{6,5},(P_{\zeta})_{6,0}=6!\left( \frac{1}{2!}\, \tau_{6,3}^2\right).
\end{equation}

\subsubsection{Nonic Calabi-Yau 7-folds}

For nonic Calabi-Yau 7-folds, the same numerical method has shown that the period matrix $P$ is also of the form
\begin{equation}
P=P_{\zeta} \cdot P_{\log}.
\end{equation}
If we define the normalized canonical period $\varpi_{M,j}$ by
\begin{equation}
\varpi_{M,j}:=\frac{1}{(2 \pi i)^j} \, \sum_{k=0}^{j} \binom{j}{k} h_k(\varphi)\,\log^{j-k} \left( 9^{-9} \varphi \right),
\end{equation}
then the linear transformation between $\varpi_M$ and $\varpi_R$ is given by
\begin{equation}
\varpi_M=P_{\log} \cdot \varpi_R.
\end{equation}
While the linear transformation between $\Pi$ and $\varpi_M$ is given by
\begin{equation}
\Pi=P_{\zeta} \cdot \varpi_M.
\end{equation}
The entries of the matrix $P_{\zeta}$ satisfy
\begin{equation}
(P_{\zeta})_{i,i}=1;~(P_{\zeta})_{i,j}=0,~\forall j>i;~(P_{\zeta})_{i,j}=\binom{i}{j}(P_{\zeta})_{i-j,0},~\forall j<i;
\end{equation}
where the indices $i$ and $j$ run from 0 to 7. Now let $\tau_{7,3}$, $\tau_{7,5}$ and $\tau_{7,7}$ be
\begin{equation}
\tau_{7,3}=-240\, \zeta(3)/(2 \pi i)^3,~\tau_{7,5}=-11\,808 \, \zeta(5)/(2 \pi i)^5,~\tau_{7,7}=-683\,280 \, \zeta(7)/(2 \pi i)^7.
\end{equation}
Our numerical results have shown that
\begin{equation}
\begin{aligned}
(P_{\zeta})_{1,0}&=(P_{\zeta})_{2,0}=(P_{\zeta})_{4,0}=0,~(P_{\zeta})_{3,0}=3!\, \tau_{7,3},\\
(P_{\zeta})_{5,0}&=5! \,\tau_{7,5}, (P_{\zeta})_{6,0}=6!\left( \frac{1}{2!}\, \tau_{7,3}^2\right), ~(P_{\zeta})_{7,0}=7! \,\tau_{7,7}
\end{aligned}
\end{equation}

\subsubsection{Decic Calabi-Yau 8-folds}

For decic Calabi-Yau 8-folds, the same numerical method has shown that the period matrix $P$ is also of the form
\begin{equation}
P=P_{\zeta} \cdot P_{\log}.
\end{equation}
If we define the normalized canonical period $\varpi_{M,j}$ by
\begin{equation}
\varpi_{M,j}:=\frac{1}{(2 \pi i)^j} \, \sum_{k=0}^{j} \binom{j}{k} h_k(\varphi)\,\log^{j-k} \left( 10^{-10} \varphi \right),
\end{equation}
then the linear transformation between $\varpi_M$ and $\varpi_R$ is given by
\begin{equation}
\varpi_M=P_{\log} \cdot \varpi_R.
\end{equation}
While the linear transformation between $\Pi$ and $\varpi_M$ is given by
\begin{equation}
\Pi=P_{\zeta} \cdot \varpi_M.
\end{equation}
The entries of the matrix $P_{\zeta}$ satisfies
\begin{equation}
(P_{\zeta})_{i,i}=1;~(P_{\zeta})_{i,j}=0,~\forall j>i;~(P_{\zeta})_{i,j}=\binom{i}{j}(P_{\zeta})_{i-j,0},~\forall j<i;
\end{equation}
where the indices $i$ and $j$ run from 0 to 8. Now let $\tau_{8,3}$, $\tau_{8,5}$ and $\tau_{8,7}$ be
\begin{equation}
\tau_{8,3}=-330\, \zeta(3)/(2 \pi i)^3,~\tau_{8,5}=-19\,998 \, \zeta(5)/(2 \pi i)^5,~\tau_{8,7}=-1\,428\,570 \, \zeta(7)/(2 \pi i)^7.
\end{equation}
Our numerical results have shown that
\begin{equation}
\begin{aligned}
(P_{\zeta})_{1,0}&=(P_{\zeta})_{2,0}=(P_{\zeta})_{4,0}=0,~(P_{\zeta})_{3,0}=3!\, \tau_{8,3}, (P_{\zeta})_{5,0}=5! \,\tau_{8,5},\\
 (P_{\zeta})_{6,0}&=6!\left( \frac{1}{2!}\, \tau_{8,3}^2\right), ~(P_{\zeta})_{7,0}=7! \,\tau_{8,7},~(P_{\zeta})_{8,0}=8! \,\tau_{8,3} \tau_{8,5}
\end{aligned}
\end{equation}

\subsubsection{Undenic Calabi-Yau 9-folds}
For undenic Calabi-Yau 9-folds, the same numerical method has shown that the period matrix $P$ is also of the form
\begin{equation}
P=P_{\zeta} \cdot P_{\log}.
\end{equation}
If we define the normalized canonical period $\varpi_{M,j}$ by
\begin{equation}
\varpi_{M,j}:=\frac{1}{(2 \pi i)^j} \, \sum_{k=0}^{j} \binom{j}{k} h_k(\varphi)\,\log^{j-k} \left( 11^{-11} \varphi \right),
\end{equation}
then the linear transformation between $\varpi_M$ and $\varpi_R$ is given by
\begin{equation}
\varpi_M=P_{\log} \cdot \varpi_R.
\end{equation}
While the linear transformation between $\Pi$ and $\varpi_M$ is given by
\begin{equation}
\Pi=P_{\zeta} \cdot \varpi_M.
\end{equation}
The entries of the matrix $P_{\zeta}$ satisfy
\begin{equation}
(P_{\zeta})_{i,i}=1;~(P_{\zeta})_{i,j}=0,~\forall j>i;~(P_{\zeta})_{i,j}=\binom{i}{j}(P_{\zeta})_{i-j,0},~\forall j<i;
\end{equation}
where the indices $i$ and $j$ run from 0 to 9. 
Now let $\tau_{9,3}$, $\tau_{9,5}$, $\tau_{9,7}$ and $\tau_{9,9}$ be
\begin{equation}
\begin{aligned}
\tau_{9,3}&=-440\, \zeta(3)/(2 \pi i)^3,~\tau_{9,5}=-32\,208 \, \zeta(5)/(2 \pi i)^5,\\
\tau_{9,7}&=-2\,783\,880 \, \zeta(7)/(2 \pi i)^7,~\tau_{9,9}=-\frac{785\,982\,560}{3} \, \zeta(9)/(2 \pi i)^9.
\end{aligned}
\end{equation}
Our numerical results have shown that
\begin{equation}
\begin{aligned}
(P_{\zeta})_{1,0}&=(P_{\zeta})_{2,0}=(P_{\zeta})_{4,0}=0,~(P_{\zeta})_{3,0}=3!\, \tau_{9,3},\\
(P_{\zeta})_{5,0}&=5! \,\tau_{9,5}, (P_{\zeta})_{6,0}=6!\left( \frac{1}{2!}\, \tau_{9,3}^2\right), ~(P_{\zeta})_{7,0}=7! \,\tau_{9,7},\\
(P_{\zeta})_{8,0}&=8! \,\tau_{9,3} \tau_{9,5},~(P_{\zeta})_{9,0}=9!\left(\tau_{9,9}+\frac{1}{3!} \tau_{9,3}^3 \right).
\end{aligned}
\end{equation}

\subsubsection{Dudecic Calabi-Yau 10-folds}
For dudecic Calabi-Yau 10-folds, the same numerical method has shown that the period matrix $P$ is also of the form
\begin{equation}
P=P_{\zeta} \cdot P_{\log}.
\end{equation}
If we define the normalized canonical period $\varpi_{M,j}$ by
\begin{equation}
\varpi_{M,j}:=\frac{1}{(2 \pi i)^j} \, \sum_{k=0}^{j} \binom{j}{k} h_k(\varphi)\,\log^{j-k} \left( 12^{-12} \varphi \right),
\end{equation}
then the linear transformation between $\varpi_M$ and $\varpi_R$ is given by
\begin{equation}
\varpi_M=P_{\log} \cdot \varpi_R.
\end{equation}
While the linear transformation between $\Pi$ and $\varpi_M$ is given by
\begin{equation}
\Pi=P_{\zeta} \cdot \varpi_M.
\end{equation}
The entries of the matrix $P_{\zeta}$ satisfy
\begin{equation}
(P_{\zeta})_{i,i}=1;~(P_{\zeta})_{i,j}=0,~\forall j>i;~(P_{\zeta})_{i,j}=\binom{i}{j}(P_{\zeta})_{i-j,0},~\forall j<i;
\end{equation}
where the indices $i$ and $j$ run from 0 to 10. 
Now let $\tau_{10,3}$, $\tau_{10,5}$, $\tau_{10,7}$ and $\tau_{10,9}$ be
\begin{equation}
\begin{aligned}
\tau_{10,3}&=-572\, \zeta(3)/(2 \pi i)^3,~\tau_{10,5}=-49\,764 \, \zeta(5)/(2 \pi i)^5,\\
\tau_{10,7}&=-5\,118\,828 \, \zeta(7)/(2 \pi i)^7,~\tau_{10,9}=-\frac{1\,719\,926\,780}{3} \, \zeta(9)/(2 \pi i)^9.
\end{aligned}
\end{equation}
Our numerical results have shown that
\begin{equation}
\begin{aligned}
(P_{\zeta})_{1,0}&=(P_{\zeta})_{2,0}=(P_{\zeta})_{4,0}=0,~(P_{\zeta})_{3,0}=3!\, \tau_{10,3}, ~(P_{\zeta})_{5,0}=5! \,\tau_{10,5}, \\
 (P_{\zeta})_{6,0}&=6!\left( \frac{1}{2!}\, \tau_{10,3}^2\right), ~(P_{\zeta})_{7,0}=7! \,\tau_{10,7}, ~(P_{\zeta})_{8,0}=8! \,\tau_{10,3} \tau_{10,5},\\
(P_{\zeta})_{9,0}&=9!\left(\tau_{10,9}+\frac{1}{3!} \tau_{10,3}^3 \right),~(P_{\zeta})_{10,0}=10! \left(\frac{1}{2!}\,\tau_{10,5}^2+\tau_{10,3} \tau_{10,7} \right).
\end{aligned}
\end{equation}

\subsubsection{Tredecic Calabi-Yau 11-folds}
For tredecic Calabi-Yau 11-folds, the same numerical method has shown that the period matrix $P$ is also of the form
\begin{equation}
P=P_{\zeta} \cdot P_{\log}.
\end{equation}
If we define the normalized canonical period $\varpi_{M,j}$ by
\begin{equation}
\varpi_{M,j}:=\frac{1}{(2 \pi i)^j} \, \sum_{k=0}^{j} \binom{j}{k} h_k(\varphi)\,\log^{j-k} \left( 13^{-13} \varphi \right),
\end{equation}
then the linear transformation between $\varpi_M$ and $\varpi_R$ is given by
\begin{equation}
\varpi_M=P_{\log} \cdot \varpi_R.
\end{equation}
While the linear transformation between $\Pi$ and $\varpi_M$ is given by
\begin{equation}
\Pi=P_{\zeta} \cdot \varpi_M.
\end{equation}
The entries of the matrix $P_{\zeta}$ satisfy
\begin{equation}
(P_{\zeta})_{i,i}=1;~(P_{\zeta})_{i,j}=0,~\forall j>i;~(P_{\zeta})_{i,j}=\binom{i}{j}(P_{\zeta})_{i-j,0},~\forall j<i;
\end{equation}
where the indices $i$ and $j$ run from 0 to 11. Now let $\tau_{11,3}$, $\tau_{11,5}$, $\tau_{11,7}$, $\tau_{11,9}$ and $\tau_{11,11}$ be
\begin{equation}
\begin{aligned}
\tau_{11,3}&=-728\, \zeta(3)/(2 \pi i)^3,~\tau_{11,5}=-74\,256 \, \zeta(5)/(2 \pi i)^5,\tau_{11,7}=-8\,964\,072 \, \zeta(7)/(2 \pi i)^7,\\
\tau_{11,9}&=-\frac{3\,534\,833\,120}{3} \, \zeta(9)/(2 \pi i)^9,~\tau_{11,11}=-162\,923\,672\,184 \,\zeta(11)/(2 \pi i)^{11}.
\end{aligned}
\end{equation}
Our numerical results have shown that
\begin{equation}
\begin{aligned}
(P_{\zeta})_{1,0}&=(P_{\zeta})_{2,0}=(P_{\zeta})_{4,0}=0,~(P_{\zeta})_{3,0}=3!\, \tau_{11,3},
(P_{\zeta})_{5,0}=5! \,\tau_{11,5}, (P_{\zeta})_{6,0}=6!\left( \frac{1}{2!}\, \tau_{11,3}^2\right), \\
(P_{\zeta})_{7,0}&=7! \,\tau_{11,7},
(P_{\zeta})_{8,0}=8! \,\tau_{11,3} \tau_{11,5},(P_{\zeta})_{9,0}=9!\left(\tau_{11,9}+\frac{1}{3!} \tau_{11,3}^3 \right),\\
(P_{\zeta})_{10,0}&=10! \left(\frac{1}{2!}\,\tau_{11,5}^2+\tau_{11,3} \tau_{11,7} \right),~(P_{\zeta})_{11,0}=11!\left(\tau_{11,11}+\frac{1}{2!}\,\tau_{11,3}^2\tau_{11,5} \right).
\end{aligned}
\end{equation}

\subsubsection{Quattuordecic Calabi-Yau 12-folds}
For quattuordecic Calabi-Yau 12-folds, the same numerical method has shown that the period matrix $P$ is also of the form
\begin{equation}
P=P_{\zeta} \cdot P_{\log}.
\end{equation}
If we define the normalized canonical period $\varpi_{M,j}$ by
\begin{equation}
\varpi_{M,j}:=\frac{1}{(2 \pi i)^j} \, \sum_{k=0}^{j} \binom{j}{k} h_k(\varphi)\,\log^{j-k} \left( 14^{-14} \varphi \right),
\end{equation}
then the linear transformation between $\varpi_M$ and $\varpi_R$ is given by
\begin{equation}
\varpi_M=P_{\log} \cdot \varpi_R.
\end{equation}
While the linear transformation between $\Pi$ and $\varpi_M$ is given by
\begin{equation}
\Pi=P_{\zeta} \cdot \varpi_M.
\end{equation}
The entries of the matrix $P_{\zeta}$ satisfy
\begin{equation}
(P_{\zeta})_{i,i}=1;~(P_{\zeta})_{i,j}=0,~\forall j>i;~(P_{\zeta})_{i,j}=\binom{i}{j}(P_{\zeta})_{i-j,0},~\forall j<i,
\end{equation}
where the indices $i$ and $j$ run from 0 to 12. Now let $\tau_{12,3}$, $\tau_{12,5}$, $\tau_{12,7}$, $\tau_{12,9}$ and $\tau_{12,11}$ be
\begin{equation}
\begin{aligned}
\tau_{12,3}&=-910\, \zeta(3)/(2 \pi i)^3,~\tau_{12,5}=-107\,562 \, \zeta(5)/(2 \pi i)^5,\tau_{12,7}=-15\,059\,070 \, \zeta(7)/(2 \pi i)^7,\\
\tau_{12,9}&=-\frac{6\,887\,015\,590}{3} \, \zeta(9)/(2 \pi i)^9,~\tau_{12,11}=-368\,142\,288\,150 \,\zeta(11)/(2 \pi i)^{11}.
\end{aligned}
\end{equation}
Our numerical results have shown that
\begin{equation}
\begin{aligned}
(P_{\zeta})_{1,0}&=(P_{\zeta})_{2,0}=(P_{\zeta})_{4,0}=0,~(P_{\zeta})_{3,0}=3!\, \tau_{12,3},
(P_{\zeta})_{5,0}=5! \,\tau_{12,5}, \\
 (P_{\zeta})_{6,0}&=6!\left( \frac{1}{2!}\, \tau_{12,3}^2\right),(P_{\zeta})_{7,0}=7! \,\tau_{12,7},
(P_{\zeta})_{8,0}=8! \,\tau_{12,3} \tau_{12,5},\\
(P_{\zeta})_{9,0}&=9!\left(\tau_{12,9}+\frac{1}{3!} \tau_{12,3}^3 \right),
(P_{\zeta})_{10,0}=10! \left(\frac{1}{2!}\,\tau_{12,5}^2+\tau_{12,3} \tau_{12,7} \right), \\
(P_{\zeta})_{11,0}&=11!\left(\tau_{12,11}+\frac{1}{2!}\,\tau_{12,3}^2\tau_{12,5} \right),~(P_{\zeta})_{12,0}=12! \left(\frac{1}{4!} \tau_{12,3}^4 + \tau_{12,5} \tau_{12,7}+\tau_{12,3}\tau_{12,9}\right).
\end{aligned}
\end{equation}

\section{Generalization to CY \texorpdfstring{$n$}{n}-folds and a motivic conjecture} \label{sec:generalizationNdimensional}

In this section, we will first generalize our results in last section to the Fermat pencil of Calabi-Yau $n$-folds, and then introduce a motivic conjecture that will explain the occurrence of the zeta values in the period matrix.

\subsection{The period matrix of CY \texorpdfstring{$n$}{n}-folds } \label{sec:periodmatrixgenerals}

Based on our computations in Section \ref{sec:numericalperiodmatrix} for the cases where $n=4,5,6,7,8,9,10,11,12$, it is straightforward to conjecture the form of the entries of the period matrix $P$ when $n \geq 4$. First, the period matrix $P$ should be of the form
\begin{equation}
P=P_{\zeta} \cdot P_{\log}.
\end{equation}
If we now define the normalized canonical period $\varpi_{M,j}$ by
\begin{equation}
\varpi_{M,j}:=\frac{1}{(2 \pi i)^j} \, \sum_{k=0}^{j} \binom{j}{k} h_k(\varphi)\,\log^{j-k} \left( (n+2)^{-(n+2)} \varphi \right),
\end{equation}
then the linear transformation between $\varpi_M$ and $\varpi_R$ is given by
\begin{equation}
\varpi_M=P_{\log} \cdot \varpi_R.
\end{equation}
The linear transformation between $\Pi$ and $\varpi_M$ is given by
\begin{equation}
\Pi=P_{\zeta} \cdot \varpi_M.
\end{equation}
The matrix $P_{\zeta}$ is $(n+1) \times (n+1) $ with entries satisfy
\begin{equation}
\left( P_{\zeta} \right)_{jk}=
\begin{cases}
0, \text{if}~ j<k, \\
1,\text{if}~ j=k, \\
\binom{j}{k}\,\left( P_{\zeta} \right)_{j-k,0}, \text{otherwise}.
\end{cases}
\end{equation}
where the indices $j$ and $k$ run from 0 to $n$. Based on our computations in Section \ref{sec:numericalperiodmatrix}, we should have
\begin{equation}
\left( P_{\zeta} \right)_{1,0}=\left( P_{\zeta} \right)_{2,0}=\left( P_{\zeta} \right)_{4,0}=0.
\end{equation}
For every odd integer $k$ such that $3 \leq k \leq n$, there exists a number $\tau_{n,k}$ of the form
\begin{equation} \label{eq:tauvalue}
\tau_{n,k}=-r_{n,k}\,\zeta(k)/(2 \pi i)^k,~r_{n,k} \in \mathbb{Q}_+,
\end{equation}
where the value of $r_{n,k}$ depends on $n$. The value of $\left( P_{\zeta} \right)_{j,0}$ is determined by the number of ways to write $j$ as the sum of odd integers that are also $\geq 3$. More precisely, let us define an odd-sum partition $\mathcal{P}$ of $j $ by
\begin{equation}
\mathcal{P}=\{ \underbrace{p_1,\cdots, p_1}_\textrm{$l_1$}, \underbrace{p_2,\cdots, p_2}_\textrm{$l_2$}, \cdots,\underbrace{p_k,\cdots, p_k}_\textrm{$l_k$} \},
\end{equation}
where each $p_m$ is an odd integer that satisfies $3 \leq p_1 <p_2 <\cdots <p_k \leq n$, and moreover
\begin{equation}
 l_1p_1+l_2p_2+\cdots l_kp_k=j.
\end{equation}
The value of $\left( P_{\zeta} \right)_{j,0}$ should be given by
\begin{equation} \label{eq:oddsumpartitionzeta}
\left( P_{\zeta} \right)_{j,0}=j! \sum_{\mathcal{P}} \left(\frac{1}{l_1!l_2!\cdots l_k!} \tau_{n,p_1}^{l_1}\tau_{n,p_2}^{l_2} \cdots \tau_{n,p_k}^{l_k} \right),
\end{equation}
where the sum is over all the odd-sum partitions of $j$. For example, $9$ has two odd-sum partitions given by
\begin{equation}
3+3+3,~9,
\end{equation}
and $\left( P_{\zeta} \right)_{9,0}$ should be
\begin{equation}
\left( P_{\zeta} \right)_{9,0}=9! \left(\frac{1}{3!} \tau_{n,3}^3 + \tau_{n,9}\right).
\end{equation}
Another example is when $n=12$, which has three odd-sum partitions given by
\begin{equation}
3+3+3+3,~3+9,~5+7,
\end{equation}
and $\left( P_{\zeta} \right)_{12,0}$ should be
\begin{equation}
\left( P_{\zeta} \right)_{12,0}=12! \left(\frac{1}{4!} \tau_{n,3}^4 + \tau_{n,5} \tau_{n,7}+\tau_{n,3}\tau_{n,9}\right).
\end{equation}

\begin{remark}
For a different family of Calabi-Yau $n$-folds that has a large complex structure limit, formula \ref{eq:oddsumpartitionzeta} should also work. 
\end{remark}

\subsection{Mixed Tate motives} \label{sec:mixedTatemotives}

Now we briefly discuss the abelian category of mixed Tate motives defined over $\mathbb{Q}$. Suppose $\textbf{DM}_{\text{gm}}(\mathbb{Q},\mathbb{Q})$ is Voevodsky's category of mixed motives, which is a rigid tensor triangulated category defined over $\mathbb{Q}$. The Tate objects, denoted by $\mathbb{Q}(n),n \in \mathbb{Z}$, generate a full triangulated subcategory. From the paper \cite{LevineVC}, there exists a motivic $t$-structure on $\textbf{DTM}_{\mathbb{Q}}$ whose heart is by definition the abelian category of mixed Tate motives $\textbf{TM}_{\mathbb{Q}}$. The readers are referred to the paper \cite{LevineVC} for more details.

Suppose $\mathcal{A}$ is an abelian category. Given two objects $A$ and $B$ of $\mathcal{A}$, an extension of $B$ by $A$ is a short exact sequence
\begin{equation} \label{eq:TMExtension}
\begin{tikzcd}
0 \arrow[r] &  A \arrow[r] & E \arrow[r] & B \arrow[r] & 0.
\end{tikzcd}
\end{equation}
Two extensions of $B$ by $A$ are said to be isomorphic if there exists a commutative diagram of the form
\begin{equation}
\begin{tikzcd}
0 \arrow[r]  &  A \arrow[r] \arrow[d,"\text{Id}"] & E \arrow[r] \arrow[d,"\simeq"] & B \arrow[r] \arrow[d,"\text{Id}"] & 0 \\
0 \arrow[r] & A \arrow[r] &  E' \arrow[r] & B \arrow[r] & 0
\end{tikzcd}.
\end{equation}
The extension \ref{eq:TMExtension} is said to split if it is isomorphic to the trivial extension
\begin{equation} \label{eq:TMTrivialExtension}
\begin{tikzcd}
0 \arrow[r] & A \arrow[r,"i"] & A \oplus B \arrow[r,"j"] & B \arrow[r]  &  0,
\end{tikzcd}
\end{equation}
where $i$ is the natural injection and $j$ is the natural projection. The set of isomorphism classes of extensions of $B$ by $A$, denoted by $\text{Ext}^1_{\mathcal{A}}(B,A)$, has a group structure induced by Baer summation with the trivial extension  \ref{eq:TMTrivialExtension} being the zero element.

Every object of the abelian category $\textbf{TM}_{\mathbb{Q}}$ can be represented as a successive extension of direct sums of Tate objects $\mathbb{Q}(n)$. It is very important that the extensions of $\mathbb{Q}(0)$ by $\mathbb{Q}(n),n \geq 3$ in $\textbf{TM}_{\mathbb{Q}}$ have an explicit description from Corollary 4.3 of \cite{LevineVC}. First, there exists a  Hodge realisation functor $\mathfrak{R}$ on $\textbf{TM}_{\mathbb{Q}}$ \cite{DeligneTate}
\begin{equation} \label{eq:HodgeRealTate}
\mathfrak{R}:\textbf{TM}_{\mathbb{Q}} \rightarrow \textbf{MHS}_{\mathbb{Q}},
\end{equation}
where $\textbf{MHS}_{\mathbb{Q}}$ is the abelian category of mixed Hodge structures. From the paper \cite{DeligneTate}, the Hodge realisation functor $\mathfrak{R}$ is exact and full-faithful, hence it induces an injective homomorphism from $\text{Ext}^1_{\textbf{TM}_{\mathbb{Q}}}\left(\mathbb{Q}(0),\mathbb{Q}(n)\right)$ to $\text{Ext}^1_{\textbf{MHS}_{\mathbb{Q}}}\left(\mathbb{Q}(0),\mathbb{Q}(n)\right)$. The latter extension group has a very simple description \cite{PetersSteenbrink}
\begin{equation} \label{eq:THexe3}
 \text{Ext}^1_{\textbf{MHS}_{\mathbb{Q}}}\left(\mathbb{Q}(0),\mathbb{Q}(n)\right) \simeq  \mathbb{C}/(2 \pi i)^n\,\mathbb{Q}.
\end{equation}
While from \cite{DeligneFund,LevineVC}, for an integer $n \geq 3$, the image of $\text{Ext}^1_{\textbf{TM}_{\mathbb{Q}}}\left(\mathbb{Q}(0),\mathbb{Q}(n)\right)$ in $ \mathbb{C}/(2  \pi  i)^n\,\mathbb{Q}$ under Hodge realisation is the coset of rational multiples of $\zeta(n)$. Suppose $M$ is a mixed Tate motive that forms an extension of $\mathbb{Q}(0)$ by $\mathbb{Q}(2k+1),k \geq 1$
\begin{equation}
\begin{tikzcd}
0 \arrow[r] &  \mathbb{Q}(2k+1) \arrow[r] & M \arrow[r] & \mathbb{Q}(0) \arrow[r] & 0,
\end{tikzcd}
\end{equation}
then its period matrix is of the form \cite{DeligneFund}
\begin{equation}
\begin{pmatrix}
1, & 0, \\
r\,\zeta(2k+1),& (2 \pi i)^{2k+1}, \\
\end{pmatrix},~r \in \mathbb{Q}.
\end{equation}

\subsection{A motivic conjecture} \label{sec:conclusionlast}

Now we are ready to give an explanation to the occurrence of zeta values in the period matrix $P$. The intuitive idea is that the splits of the pure Hodge structure in the formula \ref{eq:decompositionpureHodgestructure} and the limit MHS in the formula \ref{eq:splitsoflmhs} are `motivic'. 

\begin{conjecture} \label{conjecturemixedmotive}
Given a rational point $\varphi \in \mathbb{Q}-\{0,1 \}$, the pure motive $h^n(\mathscr{X}_\varphi)$ splits into the direct sum
\begin{equation}
h^n(\mathscr{X}_\varphi)= h^{n,a}(\mathscr{X}_\varphi) \oplus h^{n,t}(\mathscr{X}_\varphi),
\end{equation}
while the Hodge realization of the direct summand $h^{n,t}(\mathscr{X}_\varphi)$ (resp. $h^{n,a}(\mathscr{X}_\varphi)$) is the pure Hodge structure $\left(H^{n,t} (X, \mathbb{Q}),F^{p,t}_\varphi \right)$ (resp. $\left(H^{n,a} (X, \mathbb{Q}),F^{p,a}_\varphi \right)$). The pure motive $ h^{n,t}(\mathscr{X}_\varphi)$ has a limit at $\varphi=0$, $\mathbf{M}_{\text{lim}}$, that is a mixed Tate motive, whose Hodge realization is the limit MHS $ \left(H^{n,t}(X,\mathbb{Q}),W^t_q, F^{p,t}_{\text{lim}} \right)$.
\end{conjecture}

\begin{remark}
This conjecture admits a natural generalization to arbitrary algebraic families of Calabi-Yau $n$-folds that have a large complex structure limit, but to provide further examples is certainly very challenging.
\end{remark}

\section*{Acknowledgments}

The author is grateful to Shamit Kachru and Minhyong Kim for a reading of the draft and many helpful comments.

\appendix

\end{document}